\documentclass[11pt]{amsart}
\usepackage{amsmath,amsfonts,amssymb,amsthm}
\usepackage{epsfig}
\usepackage{graphics}
\usepackage{subfig} 
\usepackage{comment}
\usepackage{hyperref}
\usepackage[margin=1in]{geometry}

\newtheoremstyle{pak}{9pt}{9pt}{\itshape}{}{\bfseries}{}{.5em}{}

\theoremstyle{pak}
\newtheorem{thm}{Theorem}[section]

\newtheorem{lem}[thm]{Lemma}

\newtheorem{prop}[thm]{Proposition}

\newtheoremstyle{defin}
  {9pt}{9pt}{}{}{\bfseries}{}{.5em}{}
\theoremstyle{defin}

\newtheorem{cl}{Claim}
\newtheoremstyle{exm}
  {9pt}{9pt}{}{}{\scshape}{}{.5em}{}
\theoremstyle{exm}

\newtheoremstyle{proof}
  {}{}{}{}{\itshape}{:}{.5em}{}
\theoremstyle{proof}

\def\zz{\mathbb Z}

\def\rr{\mathbb R}

\def\ov{\overline}

\def\cC{\mathcal C}

\def\cD{\mathcal D}

\def\co{\mathcal O}

\def\<{\langle}
\def\>{\rangle}

\def\0{{\mathbf 0}}

\def\.{\hskip.06cm}

\def\T{{{\mathbf{T}}}}

\def\co-NP{\textup{co-NP}}
\def\NP{\textup{NP}}
\def\P{\textup{P}}
\def\SP{\textup{\#P}}
\def\G{\Gamma}
\def\cC{\mathcal{C}}

\def\Z{\mathbb{Z}}

\newcommand{\abs}[1]{\left\lvert#1\right\rvert}
\newcommand{\bs}{\backslash}

\newcommand{\latin}[1]{\textsl{#1}}

\newcommand{\fig}[1]{Figure~\ref{fig:#1}}
\newcommand{\problem}[1]{\textsc{#1}}

\newcommand{\problemdef}[3]{
\medskip
\begin{tabular}{ll}
\multicolumn{2}{l}{\problem{#1}}\\
\textbf{Instance:} & #2 \\
\textbf{Decide:} & #3
\end{tabular}\medskip}

\newcommand{\gadget}[1]{\textup{\textsf{#1}}}
\newcommand{\variable}{\gadget{V}} 
\newcommand{\vertex}{\gadget{V}}
\newcommand{\clause}{\gadget{C}}   
\newcommand{\splitter}{\gadget{Y}}
\newcommand{\hole}{\gadget{H}}
\newcommand{\crossover}{\gadget{X}}

\title[The complexity of generalized domino tilings]{The complexity of generalized domino tilings}

\author[Igor~Pak]{ \ Igor~Pak$^\star$}

\author[Jed~Yang]{ \ Jed~Yang$^\star$}

\thanks{\thinspace ${\hspace{-.45ex}}^\star$Department of Mathematics, UCLA, Los Angeles, CA 90095, USA; \.
\texttt{\{pak,jedyang\}@math.ucla.edu}}
\date{July 4, 1776}

\begin{document}
\date{}

\begin{abstract}
Tiling planar regions with dominoes is a classical problem in which the decision
and counting problems are polynomial.  We prove a variety of hardness results
(both \NP- and \SP-completeness) for different generalizations of dominoes in
three and higher dimensions.
\end{abstract}

\maketitle

\section{Introduction}
The study of \emph{domino tilings} is incredibly rich in its history, applications,
and connections to other fields.
A geometric version of the \emph{perfect matching} problem in grid graphs,
the problem is a special case of a general study of tilings of finite regions,
with its own set of problems and generalization directions,
very different from graph theory.
In this paper we study the decision and the counting problems
for the tilings of a given region.
Both problems have a large literature and long
history in combinatorics and computational complexity, as well as in the
recreational literature, but much of the work concentrated on tilings in the
plane.  Significantly less is known in three and higher dimensions, where
much of the intuition and many tools specific to two dimensions break down.

To summarize the results of this paper in one sentence, we show that the
classical domino tiling problems become computationally hard already in three
dimension, even when the topology of regions is restricted.  This further
underscores the fundamental difficulty of obtaining even the most basic
(positive) tiling results in higher dimensions.  We should mention that
dimension three is critically important for applications ranging from
enumerative combinatorics to probability, to statistical physics,
and to solid-state chemistry (see \latin{e.g.}~\cite{KRS,Ken-dimer,LP,Til}).

\smallskip

The (single tile) \emph{tileability problem} can be stated as follows.  Given a
tile~$T$, decide whether a region~$\G\subset\rr^d$ is tileable with copies of~$T$
(rotations and reflections are allowed).  When $T$ is a domino $[2\times 1]$,
the tileability problem can be solved in polynomial time, via the reduction
to perfect matching~\cite{LP}, and the result generalizes to any~$d\ge 2$.
On the other hand, when $T$ is a bar $[3\times 1]$, the tileability problem
is \NP-complete~\cite{BNRR} in the plane.

Our first result is on the tileability with a brick $T=[2\times2\times1]$, which we call
a \emph{slab}.  Viewed as half-cubes, slabs can be thought of as a natural
generalization of 2-dimensional dominoes.

\begin{thm} \label{thm:slab}
Tileability of $3$-dimensional regions with slabs is \NP-complete.
\end{thm}

Now, when the region $\G\subset\rr^2$ is simply connected (s.c.), the tileability problem is
simpler in many cases.  For example, when $T$ is a domino, the s.c.~tileability problem
can be solved in linear time in the area~$|\G|$,
while quadratic for general regions~\cite{Thu} (see also~\cite{Cha}).  Moreover,
when~$T$ is any rectangle, the s.c.~tileability problem is polynomial~\cite{Rem2} (see also~\cite{KK}),
as opposed to \NP-complete for general regions.
Interestingly, this phenomenon does not extend to higher dimensions.

\begin{thm} \label{thm:slab-c}
Tileability of \emph{contractible} $3$-dimensional regions with slabs is \NP-complete.
\end{thm}

For the definition of contractible regions in
higher dimension, generalizing s.c.~regions in the plane, see Section~\ref{s:def}.
Note that in the plane, finding a tile such that the tileability problem of simply
connected regions is \NP-complete remains an open problem (see Section~\ref{s:fin}).

\smallskip

The next set of results concerns \emph{counting problems}: given $\G$, compute the number
of tilings of $\G$ with copies of tile~$T$ (again, rotations and reflections are allowed).
It is well known that in the plane, the number of domino tilings of~$\G$ can be computed
in polynomial time (see \latin{e.g.}~\cite{Ken-dimer,LP}).  This is perhaps surprising, since
computing the number of perfect matchings is classically \SP-complete~\cite{Val}.
The following result is one of the historically first applications of \SP-completeness.

\begin{thm}[Valiant~\cite{Val-alg}] \label{thm:domino}
Number of domino tilings of $3$-dimensional regions is \SP-complete.
\end{thm}

Let us note that Valiant's result is strongly related,
but slightly different (see Subsection~\ref{ss:fin-val}). 
We re-prove Valiant's theorem both for completeness and as a stepping stone towards stronger results.  
The proof uses a reduction from the perfect matching problem in cubic bipartite graphs.  
This construction allows us to prove the following far-reaching extension of Theorem~\ref{thm:domino}.

\begin{thm} \label{thm:domino-c}
Number of domino tilings of \emph{contractible} $3$-dimensional regions is \SP-complete.
\end{thm}


\smallskip

We conclude with the following counting version of theorems~\ref{thm:slab} and~\ref{thm:slab-c}.
This is obtained essentially for free
by using the same construction as in the proof of the theorems.

\begin{thm} \label{thm:slab-s}
Number of slab tilings of $3$-dimensional regions is \SP-complete.
Moreover, the result holds when considering only contractible regions.
\end{thm}

Much of the rest of the paper consists of proofs of these results.  In Section~\ref{s:gen},
we present generalizations to dimensions~$d\ge 4$.  We conclude with final remarks and
open problems in Section~\ref{s:fin}, where we also continue a historical discussion
of these and related tiling results.

\bigskip

\section{Definitions and basic results}\label{s:def}

Consider $\Z^3$ as the union of closed unit cubes in~$\rr^3$ centered at the integer lattice points.
We will suggestively refer to the elements of~$\Z^3$ as \emph{cubes}.
A \emph{region} $\G$ is a finite subset of~$\Z^3$,
and is naturally identified with $\ov\G\subset\rr^3$, the union of the corresponding closed unit cubes.
We say that a region $\G\subset\Z^3$ is \emph{contractible} if $\ov\G\subset\rr^3$ is contractible and has contractible interior.
\begin{figure}[hbtp]
   \includegraphics[scale=0.3]{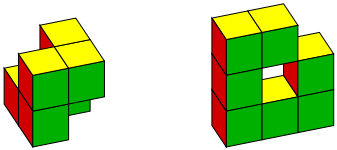}
   \caption{Examples of non-contractible regions of~$\Z^3$.}
   \label{fig:0contract}
\end{figure}
Neither of the requirements is redundant.
For example, the set on the left in \fig{0contract} consisting of six closed cubes is contractible but its interior is not, since the center becomes a hole.
The region on the right with seven closed cubes is not contractible but its interior is.
We wish to be strict and consider these as non-contractible.

Let $T$ be a \emph{tile}, which is simply a contractible region.
A \emph{$T$-tiling} of a region $\G$ is a partition of $\G$ into pairwise disjoint (as subsets of $\Z^3$) isomorphic copies of~$T$.
Here we allow $T$ to be rotated and reflected---though all the tiles we will consider are bricks, which have mirror symmetries.
When the tile $T$ is understood, we simply call it a tiling.
This leads us to the following two decision problems:

\problemdef{$T$-Tileability}
{A region~$\G$.}
{Does the region $\G$ admit a $T$-tiling?}

\problemdef{Contractible $T$-Tileability}
{A contractible region~$\G$.}
{Does the region $\G$ admit a $T$-tiling?}

\noindent
To analyze the complexities of these problems,
we will follow the usual line of attack by embedding known problems as tiling problems.
As such, let us define two more decision problems we will use.
A graph is \emph{cubic} if each vertex has degree~$3$.
A \emph{perfect matching} is a collection of pairwise-disjoint edges that cover all the vertices.

\problemdef{Cubic Bipartite Perfect Matching}
{A cubic bipartite graph~$G$.}
{Does the graph $G$ admit a perfect matching?}

\noindent
Suppose we are given a set of boolean variables.
A \emph{literal} is a variable or the negation thereof.
A \emph{$3$-SAT expression~$\cC$} is a conjunction of clauses,
each a disjunction of three literals.
We will write $\cC=\{C_1,\ldots,C_t\}$ as a set of triples $C_i$ of literals, where the negation of a variable $v$ is denoted~$\ov v$.
An assignment of boolean values to the variables is \emph{$1$-in-$3$ satisfying} if each clause $C_i$ has precisely one true literal.
This leads to the following natural decision problem:

\problemdef{$1$-in-$3$ SAT}
{A $3$-SAT expression~$\cC$.}
{Does the expression $\cC$ admit a $1$-in-$3$ satisfying assignment?}

\noindent
For each decision problem, we also have an associated counting problem,
where we count the number of witnesses.
In particular, for the problems above,
we need to count the number of tilings, perfect matchings, and satisfying assignments, respectively.
By abuse of language, we say a decision problem is \SP-complete if its associated counting problem is \SP-complete.
We reduce the tiling problems from the following two well-known problems.

\begin{thm}[\cite{DL}]
\problem{Cubic Bipartite Perfect Matching} is \SP-complete.
\end{thm}

\begin{thm}[\cite{Sch,CH}] \label{thm:1-in-3}
\problem{$1$-in-$3$ SAT} is \NP-complete and \SP-complete.
\end{thm}

For convenience and without loss of generality, we will assume that each
variable~$v_k$ appears unnegated somewhere in the boolean expression.
Indeed, if a variable occurs only negated, we may simply replace it by its negation.
If $v_k$ does not occur, add the clause $(v_k,\ov v_k,f)$, where $f$ is a new variable.
The number of satisfying assignments is obviously preserved.

\bigskip

\section{Proof of Theorem~\ref{thm:domino}} \label{sec-domino}

\subsection{Reduction construction}
We will reduce the counting of perfect matchings in cubic bipartite graphs to the counting of tilings with dominoes.
Given a cubic bipartite graph, we want to create a region,
such that the number of tilings of the region is the number of perfect matchings of the given graph.

Given two cubes $u,v\in\Z^3$,
a \emph{wire} is a sequence of distinct cubes $u=c_1,c_2,\ldots,c_t=v$ in $\Z^3$,
such that $c_i$ and $c_j$ are adjacent if and only if $\abs{i-j}=1$.
We call $c_1$ and $c_t$ the \emph{endpoints} of the wire, $\{c_2,\ldots,c_{t-1}\}$ the \emph{interior}, and $t$ the \emph{length}.
A collection of wires is \emph{proper} if no wires intersect the interior of other wires,
and any two cubes in the union of the wires are adjacent if and only if they are consecutive elements in (precisely) one such wire.

Let $G$ be a connected graph.
We will create a region $\G\subset\Z^3$ representing this graph $G$ as follows.
Let $f: V(G)\to\Z^3$ be an injective map, and identify each vertex with its image.
For each edge $uv\in E(G)$, we connect $f(u)$ and $f(v)$ by a wire, which is identified with the edge.
Let $\G$ be the union of these wires (and thus contains the vertices).
If the collection of the wires is proper, we call $\G$ a \emph{lattice drawing} of~$G$.

Color the space $\chi:\Z^3\to\Z_2$ in a checkerboard fashion,
and call $\chi(x,y,z)=x+y+z\pmod2$ the \emph{parity} of $(x,y,z)\in\Z^3$.

\begin{lem}
Every connected simple graph $G$ with maximum degree at most~$6$ has a lattice drawing.
Moreover, if $G$ is bipartite, it can be drawn with each edge having endpoints of opposite parity.
\end{lem}

\begin{proof}
It is obvious that if we pick $f$ so that the images are sufficiently spaced out,
we can connect the vertices with proper wires.
The maximum degree condition is due to the limitation that a cube has only $6$ neighbors.
The second part of the lemma is immediate.
\end{proof}

\subsection{Proof of correctness}
Fix a lattice drawing $\G$ constructed above, and consider a tiling by dominoes.
The neighborhood of a vertex is shown in \fig{1vertex}.
The solid blue square represents the vertex,
with three emanating wires representing its incident edges.

\begin{figure}[hbtp]
   \includegraphics[scale=0.3]{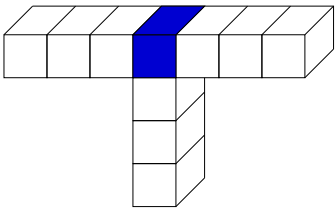}
   \caption{A vertex with three incident wires.}
   \label{fig:1vertex}
\end{figure}

Consider a wire corresponding to some edge.
Since the collection of wires is proper, there are no other cubes adjacent to the interior of this wire.
Thus in any domino tiling, the dominoes must tile the wire along its length.
Since the endpoints have opposite parity, this wire has even number of cubes.
Hence either it is partitioned into dominoes under the tiling,
or the dominoes it contains are all in its interior,
in which case its endpoints are covered by dominoes that are contained in other wires.

\begin{lem}
If $M$ is the collection of edges in $\G$ that are partitioned into dominoes,
then $M$ is a perfect matching of~$G$.
\end{lem}

\begin{proof}
Pick a vertex and consider the domino covering it: that domino is contained in precisely one of the incident edges.
Thus for every vertex precisely one incident edge is in $M$, implying that $M$ is a perfect matching.
\end{proof}

We just extracted a perfect matching of $G$ from a tiling of~$\G$.
Notice that this map is bijective, that is, every perfect matching induces a unique tiling.
Indeed, for each edge in the matching, tile the corresponding wires by dominoes covering the endpoints,
then finish the tiling in the obvious manner.
This means that the number of tilings is exactly the number of perfect matchings,
concluding the proof of Theorem~\ref{thm:domino}.

\bigskip

\section{Proof of Theorem~\ref{thm:domino-c}}

\subsection{Reduction construction}
The main idea is to create a large contractible ``plate'' that admits a unique domino tiling,
and then place the construction from the proof of Theorem~\ref{thm:domino} adjacent to this plate,
while being careful as to not introduce new tilings for the plate.
Given a region $\G$, a subregion $S\subset\G$ is called \emph{frozen} if
it is tiled the same way in all possible tilings of~$\G$.

We will call the $z=k$ plane \emph{Layer~$k$}.
The entire construction of $\G$ will lie in four layers, from Layer~$-1$ to Layer~$2$.
When considering cubes on a layer, we will employ the usual convention of referring to the $+x$, $-x$, $+y$, $-y$ directions as right, left, above, and below.
When referring to an adjacent cube in a different layer, we will specify the layer specifically, as to avoid confusion.

Let us start with a \emph{plate} that lies in Layer~$0$, which is a region whose columns have even lengths and are offset with each other,
causing \emph{jagged} borders.
\begin{figure}[hbtp]
   \includegraphics[scale=0.5]{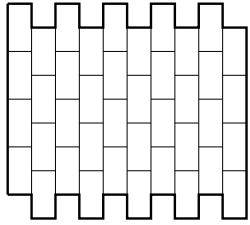}
   \caption{A plate with a unique tiling.}
   \label{fig:2plate}
\end{figure}
It is obvious that it admits a unique tiling, which is shown in \fig{2plate}.
We will modify this plate carefully as to introduce  only local changes in the tilings while preserving the unique global tiling.

Suppose we are given a cubic bipartite graph $G$ with bipartition $A$ and~$B$.
We will place the vertices in $A$ on the left side of a plate, and vertices in $B$ on the right side.
Now we must connect the corresponding vertices with wires.
We can achieve the desired connections by interchanging adjacent wires in steps.
Indeed, think of the correspondence as a permutation and write it as a product of transpositions.
The resulting schematic is known as a \emph{wiring diagram}.
These wires will, for the most part, stay in Layer~$1$.
In between each pair of wires, we will add a \emph{tension line}.
We then modify the region locally with a \emph{crossover \crossover-gadget} at each location indicated in the wiring diagram.

As an example, we see the general picture in \fig{2overview}.
\begin{figure}[hbtp]
   \includegraphics[scale=0.4]{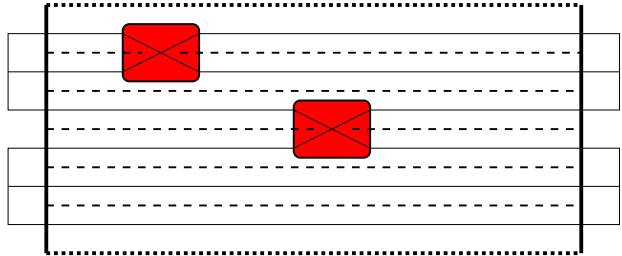}
   \caption{Overview of the layout.}
   \label{fig:2overview}
\end{figure}
There is a big plate, with offset columns causing jagged borders on the top and bottom of the region, indicated by the short dashes.
There are four vertices, each one with three wires coming out of it, indicated by the solid lines.
The tension lines between the wires are indicated by the long dashed lines.
The red boxes are the \crossover-gadgets.

The wire and the tension line are shown in \fig{2wiretension}.
\begin{figure}[hbtp]
   \subfloat[]{\label{fig:2wire}\includegraphics[scale=0.6]{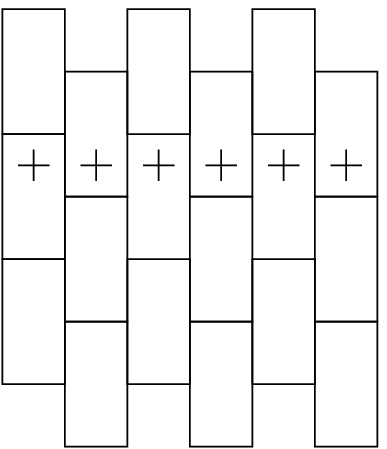}}
   \qquad
   \subfloat[]{\label{fig:2tension}\includegraphics[scale=0.6]{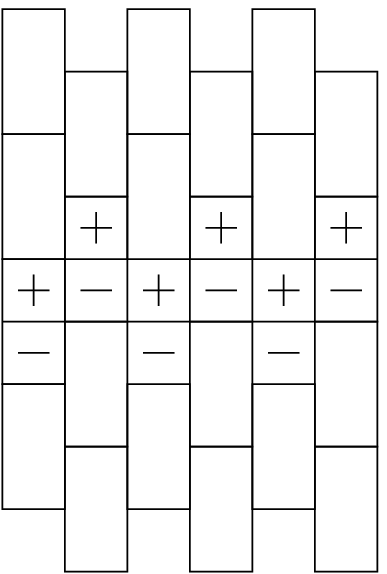}}
   \caption{A wire (left) and a tension line (right).}
   \label{fig:2wiretension}
\end{figure}
The schematic is drawn in Layer~$0$.
The unique tiling of the frozen region in Layer~$0$ is shown:
dominoes that are entirely in this layer are drawn as $2\times1$ rectangles in the plane.
When a $+$ (resp.~$-$) sign is present, it means the region includes the adjacent cube in Layer~$+1$ (resp.~$-1$).

\begin{figure}[hbtp]
   \subfloat[]{\label{fig:2vertex}\includegraphics[scale=0.6]{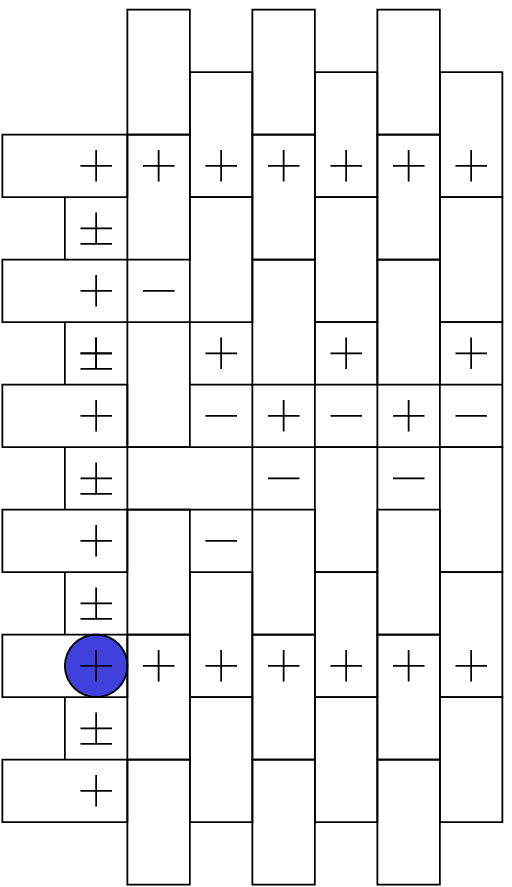}}
   \qquad
   \subfloat[]{\label{fig:2hole}\includegraphics[scale=0.6]{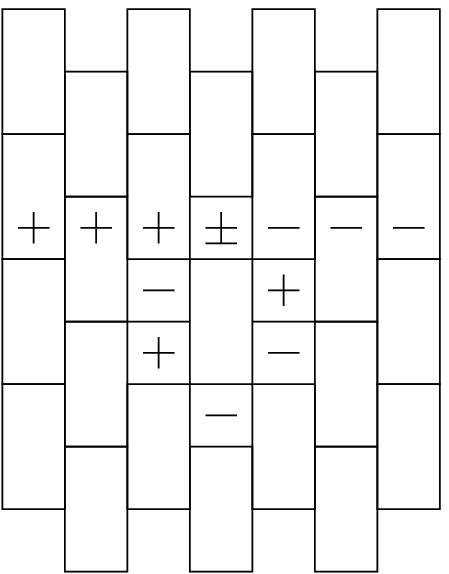}}
   \caption{Half of a vertex \vertex-gadget (left) and a hole \hole-gadget (right).}
\end{figure}

Here we show half of a vertex \vertex-gadget in \fig{2vertex}.
A single square in the schematic denotes that the cube in Layer~$0$ is always tiled with a domino sticking out to Layer~$1$ or~$-1$.
The symbol $\pm$ signifies that the adjacent cubes in both Layers $\pm1$ are included.
The blue circle is the center of the \vertex-gadget, and needs to satisfy positional parity as defined in the proof of the previous theorem.
Besides some isolated cubes, Layer~$1$ contains the wires.
Notice that there are three wires coming out of the center of the \vertex-gadget.
The third wire that is not drawn completely in the figure should mirror the other wire in the obvious way, including the initialization of the tension line.
The wires (sequences of $+$) extend towards the right.
Between each pair, we have a tension line (the sequence of jagged $+$ and~$-$).
Notice that we can easily make the \vertex-gadgets taller as to have more separation between the wires and tension lines exiting to the right.
This will be used to align all these things together correctly.
We use the mirror image when putting the \vertex-gadget on the right side of the plate.
Between pairs of \vertex-gadgets, we also add tension lines running horizontally across the entire plate.

\begin{figure}[hbtp]
   \includegraphics[scale=0.8]{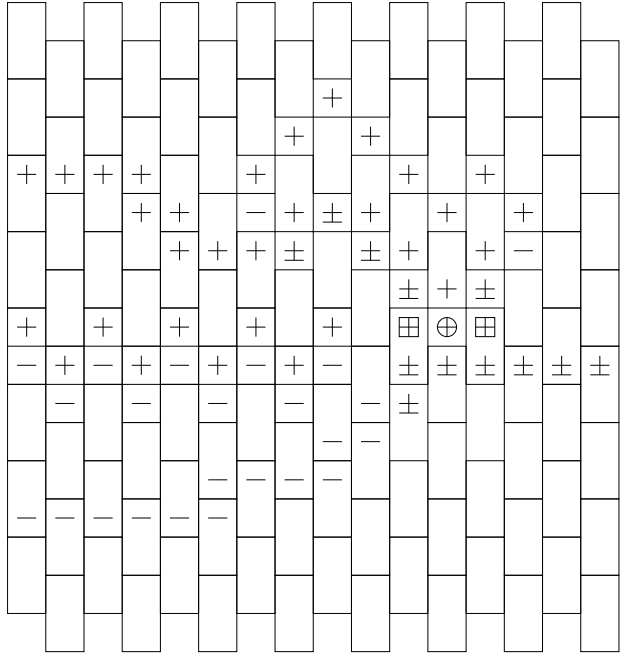}
   \caption{A splitter \splitter-gadget.}
   \label{fig:2splitter}
\end{figure}

To make the crossover \crossover-gadget, we combine the hole \hole-gadget (\fig{2hole})
and the splitter \splitter-gadget (\fig{2splitter}).
We use the same schematic convention as above in these figures.
Here the symbol~$\oplus$ denotes that there is a cube in Layer~$1$,
but not in Layer~$0$.
Similarly, symbol~$\boxplus$ signifies the inclusion
of the cubes on Layers~$1$ and~$2$ (in addition to Layer~$0$).
First use the \hole-gadget to guide the lower wire to Layer~$-1$,
and use the \splitter-gadget to merge and align them together.
Then we reverse the process using a rotated copy of the splitter,
as to make the wires crossover each other.
Finally, bring the wire back to Layer~$1$ with another \hole-gadget.
This completes the construction of the crossover \crossover-gadget.

It is clear that the \crossover-gadget takes two wires on the left, separated by a tension line,
and outputs two wires on the right, again separated by a tension line.
Notice that in the construction of the \crossover-gadget,
we could make the output wires further away from each other,
which is also the case for the \vertex-gadget, as mentioned above.
As such, we could align all these wires and tension lines so they feed into each other perfectly across the entire construction.
Thus using \crossover-gadgets, we are able to correctly connect the wires coming from \vertex-gadgets.
This complete the construction of the region, which we shall call~$\G$.

\subsection{Proof of correctness}
It remains to check that $\G$ is contractible,
and its domino tilings are in bijection with perfect matchings in graph~$G$.

\begin{lem}
The region $\G$ is contractible.
\end{lem}

\begin{proof}
Consider Layer~$0$.
It has a hole for each \splitter-gadget.
However, the adjacent cube in Layer~$1$ is present, and is contained in a $3\times3$ region.
Furthermore, the adjacent cube in Layer~$-1$ is not present.
As such, filling in the hole in Layer~$0$ does not alter contractibility.
Now that there are no holes in Layer~$0$, we may deformation retract everything in other layers to Layer~$0$, and then contract it to a point.
\end{proof}

It remains to find a frozen subregion $\G'\subset\G$ such that $\G_0=\G\bs\G'$ is a lattice drawing of the given graph.
Let us start with $\G'=\varnothing$ (thus $\G_0=\G$) and alter $\G'$ inductively.
A cube in a region is \emph{isolated} if it has precisely one neighbor (out of the six adjacent possibilities).
Notice that an isolated cube must be tiled by a domino covering its unique neighbor.
We will \emph{remove} the isolated cubes of $\G_0$ by moving them from $\G_0$ to~$\G'$.
We repeat this process inductively until there are no more isolated cubes.

\begin{lem}
At any step, $\G'$ is a frozen subregion of~$\G$.
When the process finishes, $\G_0$ is a lattice drawing of the given graph.
\end{lem}

\begin{proof}
It follows immediately by construction that $\G'$ is frozen.
The rest of this proof is devoted to the second part of the lemma.

Let us perform another induction on the number of \crossover-gadgets.
For the base case, suppose we did not have any \crossover-gadgets.
That is, we start with the boring cubic graph where disjoint pairs of vertices are joined by three edges each.

Consider each \vertex-gadget.
After the first step, everything in Layers~$0$ and $-1$ on the two left-most columns in \fig{2vertex} are removed,
leaving the wire in Layer~$1$.
The tension lines are all removed as well.
It is clear that in the next step, the remaining horizontal domino will also be removed, along with the domino above it.

Thus after two steps, what we are left with in Layer~$0$ is a collection of disconnected plates,
each of which has a wire across it horizontally in Layer~$1$.
These wires protrude to the left and right.

Each of these plates have jagged columns thanks to the initial boundary and the tension lines.
As such, in each step, the top- and bottom-most cubes are isolated, hence removed with their neighbors, creating another jagged boundary.
This process continues until all cubes in Layer~$0$ are removed, including the ones adjacent to the wire in Layer~$1$.
Note that half of the cubes adjacent to said wire are removed from above, and the other from below.
Thus it is important that both sides were jagged,
which is the motivation for introducing the tension lines.
We are left with the lattice drawing of the aforementioned contrived cubic graph.

Now for the inductive step, suppose we add an extra \crossover-gadget.
We want to show that the global situation is not altered, and that the local situation is what we wanted.
Indeed, the \crossover-gadget consists of two \hole-gadgets and two \splitter-gadgets.
Having analyzed the \vertex-gadget,
one should be able to inspect the \hole-gadget construction in \fig{2hole},
and easily see that after finitely many steps, we will be left with the wires in
Layers~$\pm1$, plus the connecting cube between them in Layer~$0$.

The \splitter-gadget is a bit more complicated, but can be carefully analyzed in the same manner.
Indeed, after the first step,
all the cubes in Layer~$-1$ are removed except for the wire.
Similarly, the Layers~$1$ and $2$ cubes for $\boxplus$ are removed, so are all the cubes in Layer~$1$, save the wire.
As for Layer~$0$, we get a small island of six vertical dominoes in the middle,
and a giant plate with enough tension lines such that a similar inductive removal process as above will remove it completely.
Hence we are left with precisely the wires, thus completing the proof of the lemma.
\end{proof}

\bigskip

\section{Proof of Theorem~\ref{thm:slab} and the first part of Theorem~\ref{thm:slab-s}}

\subsection{Reduction construction}
Here we reduce \problem{$1$-in-$3$ SAT} to a tiling problem.
The reduction will be parsimonious, thus proving both \NP-completeness and \SP-completeness.

Consider a sequence of cubes in the $z=0$ plane,
where two cubes are adjacent if and only if they are consecutive elements in the sequence
(this is called a wire in the proof of Theorem~\ref{thm:domino}).
Consider the sequence as a region and duplicate it also in the $z=1$ plane.
We now call a \emph{wire} the union of these two copies.
In a wire, a pair of adjacent cubes, one with $z=0$ and one with $z=1$, is called a \emph{cube pair}.
In any tiling by slabs, each cube pair must be covered by the same slab.
If we translate the wire to the $z=k$ and the $z=k+1$ planes, we say the wire \emph{lives} in the $z=k$ \emph{biplane}.
Color the biplane $\chi(x,y,z)=x+y\pmod2$ in a checkerboard fashion, ignoring the $z$-coordinate.
We call the color of a cube (pair) its \emph{parity}, which can be \emph{even} or \emph{odd}.
Similarly, we may rotate and consider wires living in $x=k$ or $y=k$ biplanes.
In those cases, the biplane coloring would ignore the $x$- or the $y$-coordinate, respectively.
Notice that a cube $(x,y,z)\in\Z^3$ might have different parities depending on the biplane being considered.
However, each cube pair lives in a unique biplane: thus the parity of a cube pair is unambiguous.
When we draw diagrams consisting of wires in each biplane, we simply draw the two-dimensional regions used to form the wires.


The variable \variable-gadget, shown in \fig{3variable}, is made by removing two diagonal cubes of a $2\times2\times2$ region,
and then attaching six wires.
\begin{figure}[hbtp]
   \includegraphics[scale=0.3]{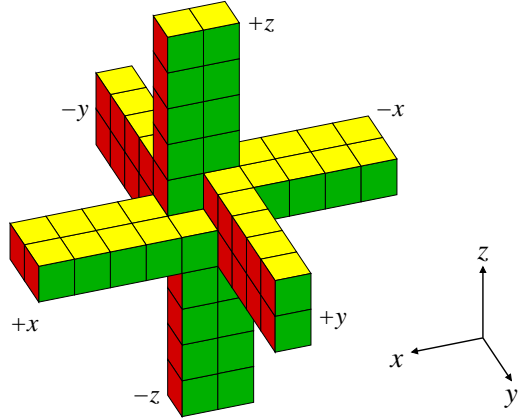}
   \caption{A variable \variable-gadget, with six wires coming out of it.}
   \label{fig:3variable}
\end{figure}
Notice that the wires come in pairs, and each pair lives in a biplane.
We label the wires with their initial coordinate directions.
Each wire has an endpoint in the \variable-gadget, called the \emph{variable endpoint}; the other one is called the \emph{free endpoint}.
We now describe where to place the \variable-gadgets and how to join these free endpoints together.

Suppose we are given a $3$-SAT expression $\cC=\{C_1,\ldots,C_t\}$ on the variable set $X=\{v_1,\ldots,v_n\}$,
where each clause $C_i=(c_{i1},c_{i2},c_{i3})$ consists of three literals,
each either a variable $v_k$ or a negation $\ov v_k$, for some $k\in[n]=\{1,\ldots,n\}$.
We now construct a region~$\G$.

If $c_{ij}$ is $v_k$ or its negation $\ov v_k$, place a \variable-gadget at $(10k, 10(3i+j), 0)$.
That is, the wires live in the $x=10k$, $y=10(3i+j)$, and $z=0$ biplanes, respectively.
Moreover, the variable endpoint of each wire has odd parity with respect to the coloring of the biplane in which the wire lives.

Consider the $x=10k$ plane for each $k\in[n]$, which is drawn schematically in \fig{3sync}.
\begin{figure}[hbtp]
   \includegraphics[scale=0.5]{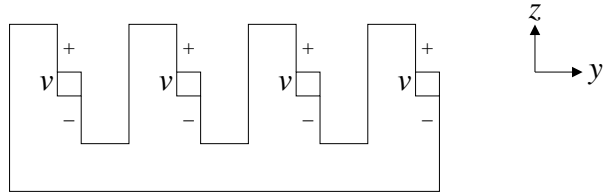}
   \caption[Synchronization of variables.]{Synchronization of variables: This is a diagram in some $x=10k$ (bi)plane. The squares represent the \variable-gadgets and the lines represent the wires.}
   \label{fig:3sync}
\end{figure}
Notice that all the $\pm z$ wires corresponding to the variable $v_k$ are on this biplane.
For each \variable-gadget, we will bend the $+z$ wire to the left and down as shown.
We link up the first wire with the last, and also the remaining adjacent pairs.

\fig{3clause} shows a clause \clause-gadget, which is simply a joining of three wires.
\begin{figure}[hbtp]
   \includegraphics[scale=0.3]{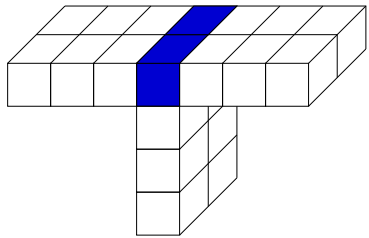}
   \caption{A clause \clause-gadget, where three wires meet.}
   \label{fig:3clause}
\end{figure}
The parity of the gadget is the parity of the blue cube pair at the center where the three wires meet.

Now let us work in the $z=0$ biplane, schematically shown in \fig{3link}.
\begin{figure}[hbtp]
   \includegraphics[scale=0.5]{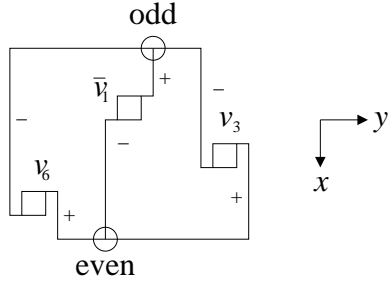}
   \caption[Diagram in the $z=0$ (bi)plane.]{Diagram in the $z=0$ (bi)plane. The circles, where the lines meet, indicate the \clause-gadgets, which are at two different parities.}
   \label{fig:3link}
\end{figure}
For each clause, there are three corresponding \variable-gadgets.
From each \variable-gadget, take the $+y$ (resp.~$-y$) wire if the corresponding variable appears in the clause unnegated (resp.\ negated).
Extend these three wires to $x=10(n+1)$ and join them together at a \clause-gadget with even parity.
Similarly, extend the three remaining $\pm y$ wires to $x=-10$ and join together at a \clause-gadget with odd parity.
We call these the even and odd \clause-gadgets corresponding to each clause, respectively.

Finally, for each \variable-gadget, let us join the remaining $+x$ wire and the $-x$ wire together to form a loop.

\begin{lem}
The wires can be joined together such that no cubes are adjacent unless they are in fact adjacent cubes from a single wire.
\end{lem}

\begin{proof}
As the \variable-gadgets are sufficiently spaced out, it is very obvious that we can make the wires avoid each other,
besides when they meet at the \variable- and \clause-gadgets.
\end{proof}

This finishes the construction.

\subsection{Proof of correctness}
Fix a region $\G$ constructed above from a $3$-SAT expression~$\cC$.
It remains to show that the tilings of~$\G$ and the $1$-in-$3$ satisfying assignments of~$\cC$ are in bijective correspondence.

First, consider a tiling of~$\G$.
We say that the \emph{phase} of an endpoint of a wire is \emph{positive}
if the wire contains the tile that covers that endpoint, and \emph{negative} otherwise.

\begin{lem} \label{lem:wire}
A wire of even length has the same phase at both endpoints,
while a wire of odd length has opposite phases at its two endpoints.
\end{lem}

\begin{proof}
As in the proof of Theorem~\ref{thm:domino}, the slabs must line up along the length of the wire,
since the various wires do not intersect or come near each other (except at the endpoints).
\end{proof}

The following is immediate by observing the way slabs can come together to tile the \variable-gadget:
\begin{lem} \label{lem:variable}
There are two ways to tile each \variable-gadget locally.
In either tiling, the three $+$~wires have the same phase at their variable endpoints,
and the three $-$~wires have the opposite phase at their variable endpoints as the $+$~wires.
\end{lem}

Let the \emph{phase} of the \variable-gadget be the common phase of the $+$~wires at their variable endpoints.

\begin{lem}
The \variable-gadgets corresponding to $v_k$ (or its negation $\ov v_k$) all have the same phase.
\end{lem}

\begin{proof}
Consider two adjacent \variable-gadgets in the construction shown in \fig{3sync}.
Suppose the gadget on the left has positive phase (the opposite situation is similar).
Then by Lemma~\ref{lem:variable}, its $-z$ wire has negative phase at the variable end.
This wire is linked with the $+z$ wire of the gadget on the right.
Since both endpoints have odd parity,
the wire has odd length,
thus by Lemma~\ref{lem:wire}, the variable endpoint of the $+z$ wire of the gadget on the right (hence the gadget) has positive phase.
\end{proof}

This allows us to assign a truth value to the variable~$v_k$:

\begin{lem}
Let the boolean assignment of $v_k$ be the common phase of the \variable-gadgets corresponding to~$v_k$.
This defines a $1$-in-$3$ satisfying assignment.
\end{lem}

\begin{proof}
Fix a clause $(c_1,c_2,c_3)$ and consider the corresponding even \clause-gadget.
There are three wires connecting the corresponding \variable-gadgets to the center of the \clause-gadget.
The variable endpoints of the wires have odd parity,
and the free endpoints are all at the center of the \clause-gadget, which has even parity.
As such, the wires have even lengths;
thus by Lemma~\ref{lem:wire}, each wire has the same phase on both its endpoints.

The center of the \clause-gadget is tiled by some slab, which is contained in precisely one of its three wires.
Without loss of generality, suppose only the wire corresponding to $c_1$ is of positive phase (at both of its endpoints).
Recall that the phase of a \variable-gadget is the phase of its $+$~wires at their variable endpoints.
If $c_i=v_k$ is an unnegated variable, then the $+y$ wire was used, and thus~$v_k$ (and hence~$c_i$) is positive if and only if~$i=1$.
On the other hand, if $c_i=\ov v_k$ is a negated variable, then the $-y$ wire was used.
Thus $v_k$ is \emph{negative} if and only if $i=1$, but then $c_i$ is again positive if and only if~$i=1$.
\end{proof}

This shows that the assignment is indeed $1$-in-$3$ satisfying.
We thus have a map from the tilings of $\G$ to the satisfying assignments of~$\cC$.

For the inverse, suppose we have a $1$-in-$3$ satisfying assignment.
Simply tile each \variable-gadget according to the assignment in the obvious way.
It only remain to tile the wires.
However, by construction, once we choose the phase of a \variable-gadget,
the tiling of all six wires emanating from that region is forced.

The $\pm z$ wires fit together because the phases of the \variable-gadgets are consistently assigned,
in accordance to the truth value of the variable in the given satisfying assignment.

Recall that for each \variable-gadget, its $+x$ and $-x$ wires form a loop.
Since the two endpoints both have odd parity, the length of the wire is odd.
By Lemma~\ref{lem:wire}, the two endpoints have opposite phases in a tiling,
which is precisely what the \variable-gadget would produce as in Lemma~\ref{lem:variable}.

It is obvious that the $\pm y$ wires fill up the even \clause-gadgets perfectly because the boolean assignment was $1$-in-$3$ satisfying.
For the odd \clause-gadgets, notice that since we take the leftover wires,
precisely one of the three is \emph{negative} at its variable endpoint.
However, since we join them at a cube pair with odd parity, the wires are of odd lengths,
thus precisely one of the three is positive on the free endpoint at the \clause-gadget, as needed.

This establishes the desired bijective correspondence,
and implies that the construction is indeed a parsimonious reduction
from \problem{$1$-in-$3$ SAT} to the slab tiling problem,
concluding the proof of Theorem~\ref{thm:slab}.
Note that this also implies that the associated counting problem is \SP-complete,
yielding the first part of Theorem~\ref{thm:slab-s};
in the next section we prove a stronger result.

\bigskip

\section{Proof of Theorem~\ref{thm:slab-c} and the second part of Theorem~\ref{thm:slab-s}}

\subsection{Reduction construction}
Given a region $\G_0$ as constructed from the previous section,
we will enlarge it to a contractible region $\G\supset\G_0$,
such that $\G'=\G\bs\G_0$ is a frozen subregion of~$\G$.

\begin{figure}[hbtp]
   \subfloat{\label{fig:4var-2c}\includegraphics[scale=0.25]{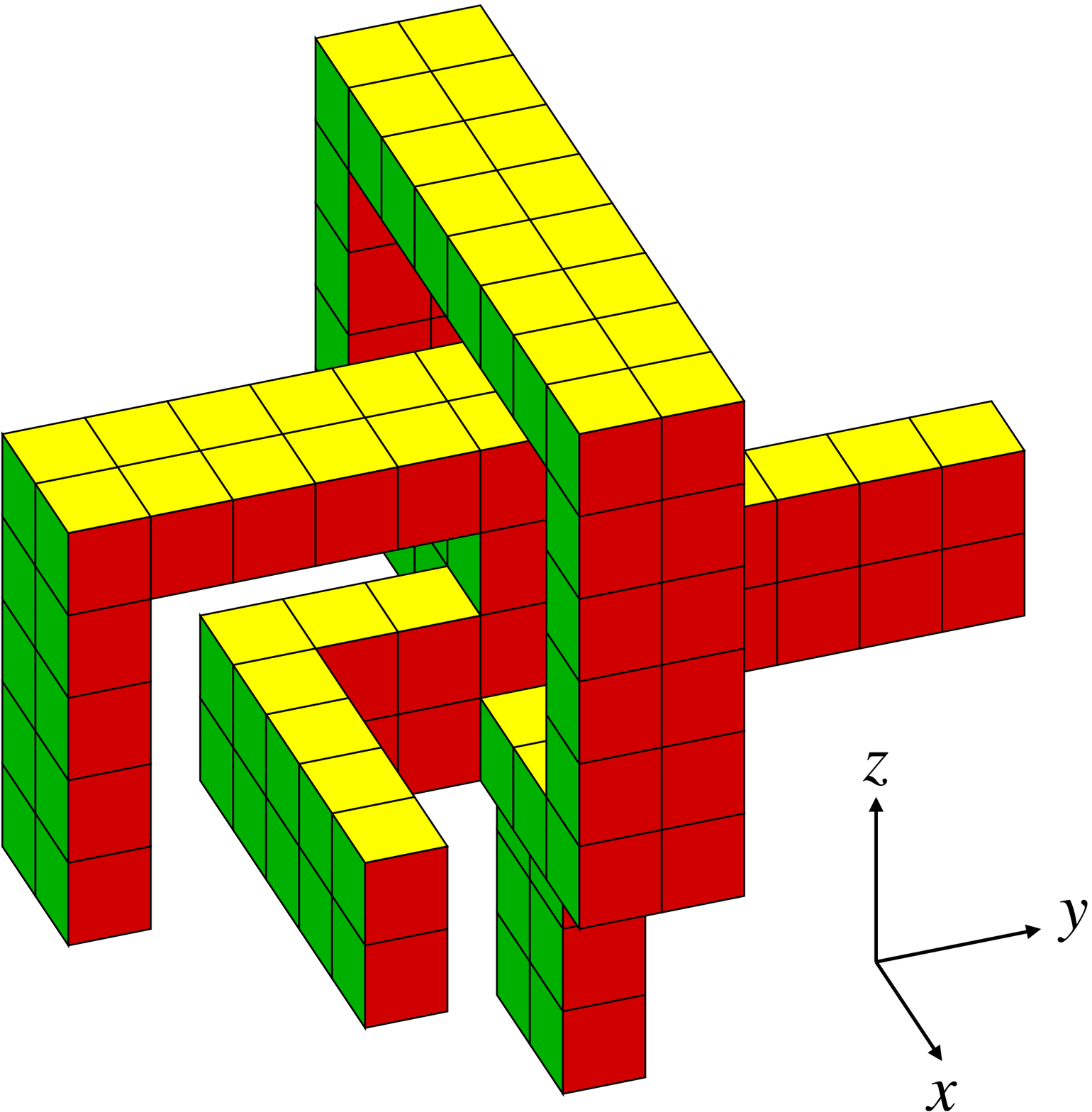}}
   \qquad
   \qquad
   \qquad
   \subfloat{\label{fig:4var-4c}\includegraphics[scale=0.25]{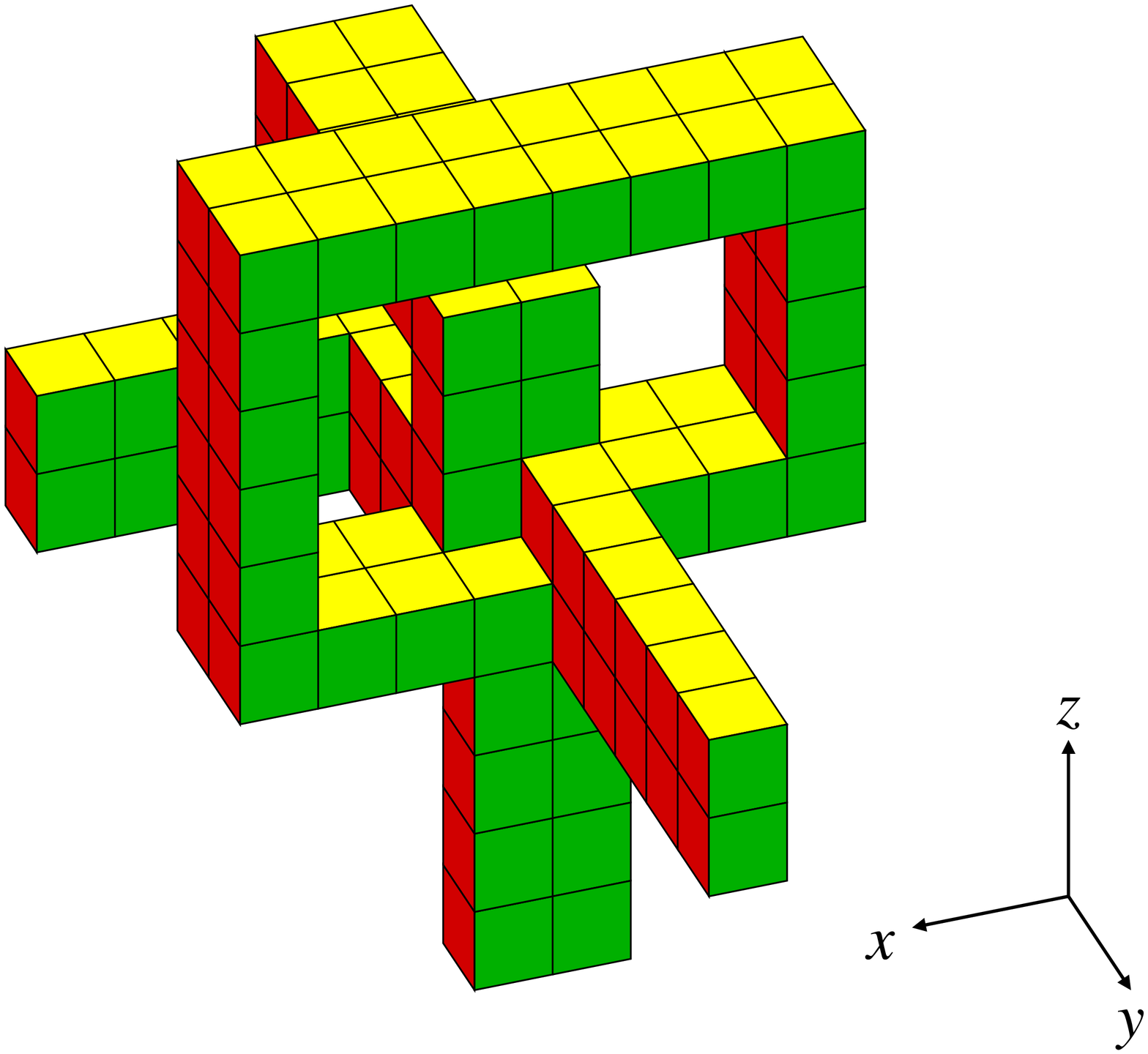}}
   \caption{A variable \variable-gadget, shown in two perspectives.}
   \label{fig:4var}
\end{figure}

First we modify the variable \variable-gadget.
The following description is relative to a \variable-gadget placed at $(0,0,0)$.
We must make these adjustments to each \variable-gadget used.
Recall that in the proof of Theorem~\ref{thm:slab}, each $+z$ wire was bent as in \fig{3sync}.
Let us make this precise by making it bend towards the $-y$ direction when it hits $z=3$,
and bend towards $-z$ when it hits $y=-5$ (see \fig{4var}).
We define a region $X$, shown in \fig{4x}, that will fill the hole this loop makes.
\begin{figure}[hbtp]
   \subfloat[Region $X$.]{\label{fig:4x}\includegraphics[scale=0.3]{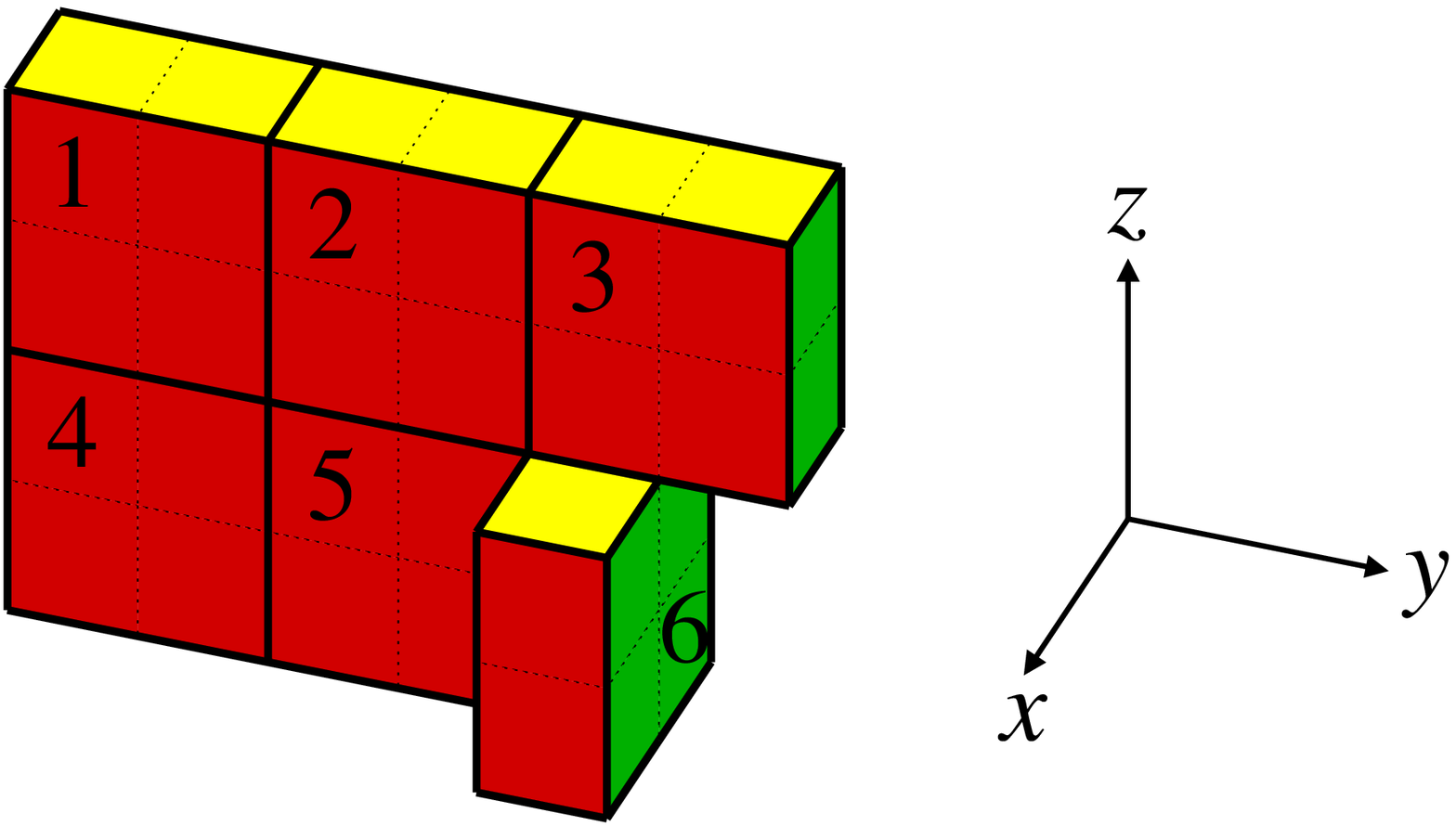}}
   \qquad\qquad
   \subfloat[Region $Y$.]{\label{fig:4y}\includegraphics[scale=0.3]{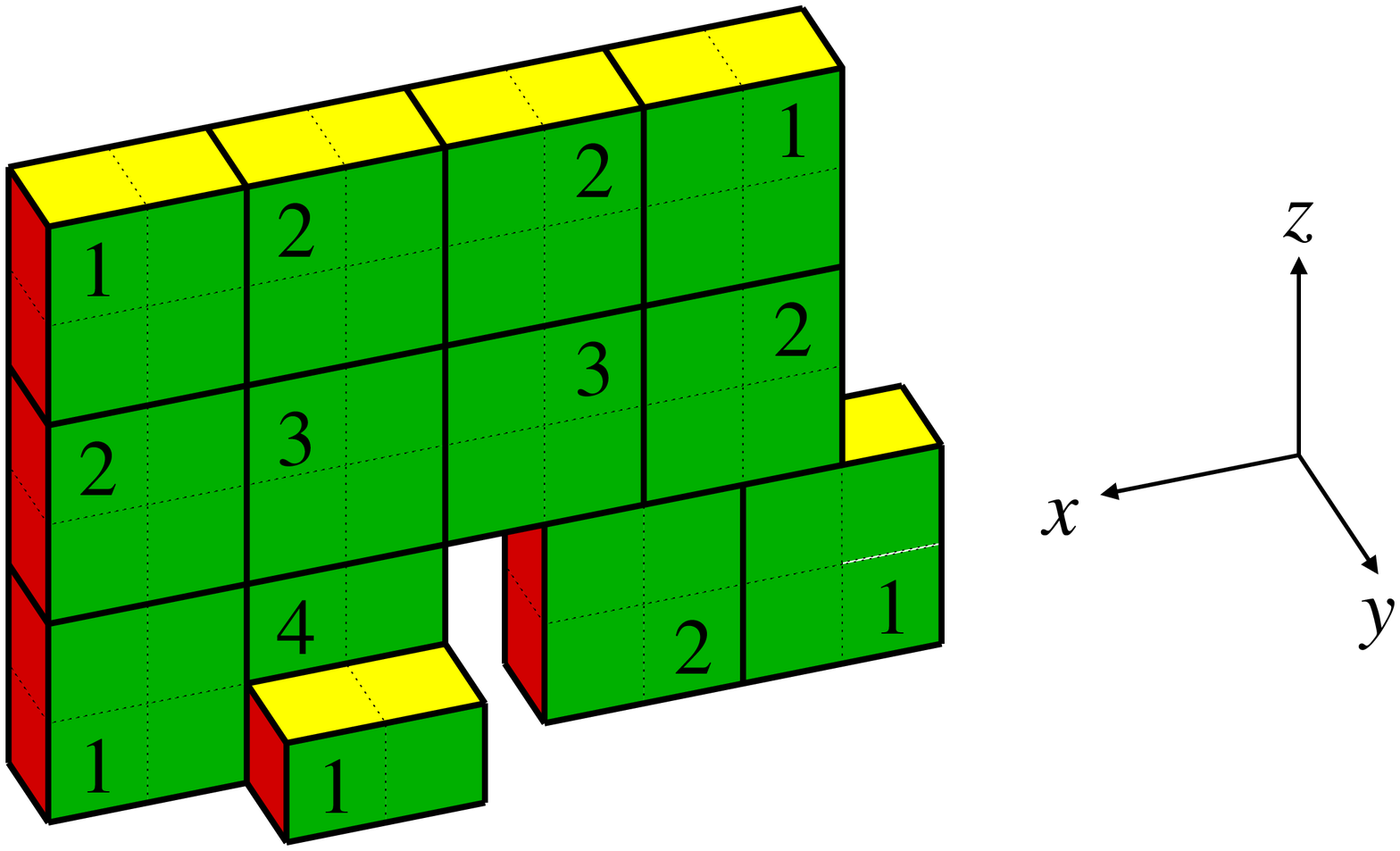}}
   \caption{Regions used to modify the \variable-gadget.}
   \label{fig:4xy}
\end{figure}
Indeed, let $X=\{-1\}\times[-5,0]\times[0,3]\cup\{(0,-1,0),(0,-1,1)\}\bs\{(-1,0,0),(-1,0,1)\}$.
The unique tiling is indicated in the figure; the labels will be used in the proof of Lemma~\ref{lem:4frozen}.
Similarly, the $\pm x$ wires are linked together to form a loop.
Let us bend them towards the $+z$ direction when they hit $4$ and $-3$, respectively,
and bend towards each other when they hit $z=5$.
Define a region that will fill this hole, shown in \fig{4y}:
Take $$[-3,4]\times\{2\}\times[0,5],$$
add the following six cubes $$(-4,2,0),(-4,2,1),(1,3,0),(2,3,0),(1,1,1),(2,1,1),$$
and remove these two cubes $$(0,2,0),(0,2,1).$$
Call this region~$Y$.
Note that the slab labeled~$4$ in the figure extends behind and is hidden from view.
We will add regions $X$ and $Y$ to the \variable-gadget.
This makes the \variable-gadget contractible.

Let $Z=[-10,10n+11]\times[0,30(t+1)+1]\times\{-1\}\subset\Z^3$,
where $n$ and $t$ are the number of variables and clauses, respectively.
We then modify $Z$ with a \emph{hole \hole-gadget} around each \variable-gadget to allow its $\pm z$ wires to pass through.
\begin{figure}[hbtp]
   \includegraphics[scale=0.3]{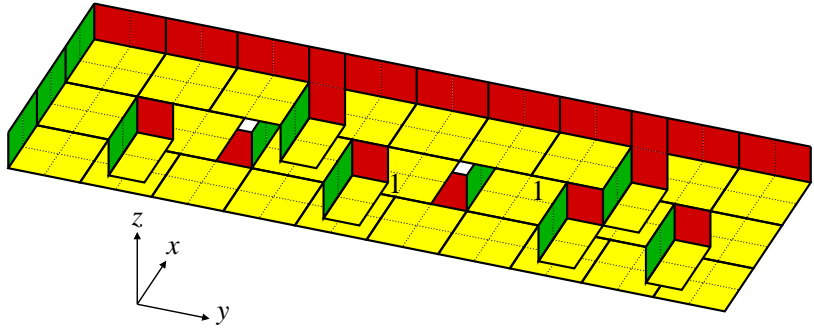}
   \caption{A hole \hole-gadget.}
   \label{fig:4hole}
\end{figure}
Indeed, \fig{4hole} shows the construction in the $z=-1$ plane viewed from below.
The surrounding slabs on the boundary match that of the unique tiling of~$Z$.
The two cube pairs missing allow the $\pm z$ wires to pass through
(spaced out according to the precise construction of the \variable-gadget in \fig{4var}).
The cube pairs in $z=-2$ plane are necessarily tiled with its unique neighbor in the $z=-1$ plane.
Let $X_k=\{10k-1\}\times[0,30(t+1)+1]\times[-9,-2]$ for each $k\in[n]$.
Here we assume that the linking of the $\pm z$ wires (see \fig{3sync}) are done at $z$-coordinates, say, $-5$ and $-7$, which are in $[-9,-2]$.
Let $\G'$ be the disjoint union of the $X$ and $Y$ for each \variable-gadget, the region $Z$, and the $X_k$ for each~$k$.
This finishes the construction.

\subsection{Proof of correctness}
It remains to check that that the construction works.  The following two lemmas easily imply the result.

\begin{lem}
The region $\G$ constructed above is contractible.
\end{lem}

\begin{proof}
Notice that there are no holes in the $z=-1$ plane.
In $z<-1$, we may deformation retract the wires onto the $X_k$ plates,
and then deformation retract these plates onto the $z=-1$ plane.

We now check that the modified \variable-gadget is contractible.
The center of an original \variable-gadget was an example of a
non-contractible region shown in \fig{0contract}.
Indeed, the interior will leave a point hole in the middle;
however, it is filled by the slab labeled~$4$ in \fig{4y}.
The loop formed by the $+z$ wire bending down is filled by~$X$,
while the loop formed by joining the $\pm x$ wires is filled by~$Y$.
Moreover, the $\pm y$ wires are all lying along $Z$ so we may
deformation retract everything to the $z=-1$ plane.
We omit the (easy) details.
\end{proof}

\begin{lem}\label{lem:4frozen}
The subregion $\G'\subset\G$ is frozen.
\end{lem}

\begin{proof}
Consider a cube $c\in\G$.
If it forms a $3\times1\times1$ region with a cube pair,
then no slab containing this cube may contain its neighbor that is part of the cube pair.
Not counting such cube pairs, if $c$ only has two neighbors left,
then $c$ must be tiled by a slab covering both these neighbors.
If this happens, we call $c$ \emph{isolated} and say that it \emph{forces} the unique slab containing~$c$.
Just like in the proof of Theorem~\ref{thm:domino-c},
we may inductively remove these forced slabs, which are obviously frozen.

Consider region~$X_k$.
Notice that the cube at the corner $(10k-1,1,-9)$ is isolated.
After removing the slab forced by it, we may consider the new corner next to it at $(10k-1,3,-9)$.
Inductively, we may remove the cubes with $z=-9$ or $-8$.
Now when looking at $z=-7$, we may run into cubes that neighbor $\pm z$ wires.
However, these wires are made of cube pairs that form $3\times1\times1$ regions with the cubes in question.
As such, we may continue the removal process unhindered.
It is clear that in this manner, we may remove all cubes from~$X_k$.

Similarly, we remove most slabs from $Z$, starting from the boundary and working our way in.
The ones sticking down out of the $z=-1$ plane can also be removed,
leaving us with precisely two slabs around the hole for the $-z$ wire from the \hole-gadget (labeled with $1$ in \fig{4hole}).
As for each \variable-gadget, we remove all the cubes from $X$ and~$Y$.
In \fig{4xy}, the forced slabs are labeled with a sample removal order,
where each number is positioned on the isolated cube used at each step.
Finally, we may remove the leftover slabs from~$Z$.
\end{proof}

Now, given a \problem{$1$-in-$3$ SAT} expression~$\cC$,
we may take a region $\G_0$ constructed in the proof of Theorem~\ref{thm:slab}
and enlarge it to a contractible region $\G$ as described above.
Since $\G'=\G\bs\G_0$ is frozen, $\G$ is tileable if and only if $\G_0$ is tileable.
These are tileable if and only if $\cC$ is $1$-in-$3$ satisfiable, thus Theorem~\ref{thm:slab-c} is proved.
Moreover, since both reductions are parsimonious,
we conclude the second counting result in Theorem~\ref{thm:slab-s} as well.

\bigskip

\section{Generalized dominoes in higher dimensions}\label{s:gen}
Let $\cD_r^d\subset\Z^d$ be a $2\times\ldots\times2\times1\times\ldots\times1$
region with $r$ number of $2$'s and $d-r$ number of~$1$'s.
We call this a \emph{generalized domino} in $d$ dimensions of rank~$r$.
Call the $r$ coordinate directions that are $2$ cubes wide \emph{fat}.
Previously we were concerned with tiling by dominoes $\cD_1^3$ and slabs $\cD_2^3$ in three dimensions.

\begin{thm}\label{t:gen}
For $2\leq r<d$, tiling (contractible) $d$-dimensional regions with $\cD_r^d$ is \NP-complete.
Similarly, for $1\leq r<d$ and $d\geq3$, tiling (contractible) $d$-dimensional regions with $\cD_r^d$ is \SP-complete.
\end{thm}

Note that in other cases of $r$ and~$d$, these problems are in~\P{} (see next section).

\begin{proof}
We reduce specific tiling problems in $\Z^3$ to higher dimensions.
Let us first prove the statements without the contractibility constraint.

Suppose $2\leq r<d$.
Let $\G\subset\Z^3$ be a region as afforded by the proof of Theorem~\ref{thm:slab}.
Let $\G'=\G\times\{0,1\}^{r-2}\times\{0\}^{d-r-1}$,
that is, $\{(x_1,\ldots,x_d)\in\Z^d:(x_1,x_2,x_3)\in\G,\ x_4,\ldots,x_{r+1}\in\{0,1\},\ x_{r+2}=\ldots=x_d=0\}$.
Consider a tiling of $\G'$ by $\cD_r^d$.
Notice that there are no $2\times2\times2$ subregion in $\G$,
thus each tile must be oriented with two of its fat directions in the first three coordinates,
the remaining $r-2$ fat directions in the next $r-2$ coordinates.
This means that any $\cD_r^d$-tiling of $\G'$ induces a $\cD_2^3$-tiling of $\G$.
It is easy to see that this correspondence is actually a bijection,
finishing the proof of \NP-completeness, and also \SP-completeness when $r\geq2$.

The case of $r=1$ is straightforward.
Take $\G$ from the proof of Theorem~\ref{thm:domino}.
Define $\G'=\G\times\{0\}^{d-3}$, and follow the argument above.

\smallskip

To prove the result for contractible regions, proceed in the same manner,
and take the regions constructed in the proof of the corresponding theorems.
For $2\leq r<d$, take the region $\G$ from the proof of Theorem~\ref{thm:slab-c},
and define $\G'$ in the same way as above.
Notice, however, that the above strategy does not yield a reduction
from $\cD_2^3$-tilings to $\cD_r^d$-tilings in general.
Indeed, we used the fact that the specific $\G\subset\Z^3$ did not contain
$2\times2\times2$ subregions, which is no longer the case.
The outlined argument actually still works, but more care must be taken.
A tile might now be oriented with three of its fat directions in the first three coordinates.
This means a $\cD_r^d$-tiling of $\G'$ will, \latin{a priori},
induce a tiling of $\G$ with slabs and the $[2\times2\times2]$ cube.
However, in the proof of Theorem~\ref{thm:slab-c},
we see that no tilings of $\G$ by slabs contain a $2\times2\times2$ subregion tiled by two slabs.
As such, all $\cD_r^d$ tiles in $\G'$ are still constrained to be oriented as before,
with precisely two fat directions in the first three coordinates.
This implies the result in this case.

The same approach works for the $r=1$ case as well.
Take $\G$ as in the proof of Theorem~\ref{thm:domino-c},
and define $\G'$ as above.
Now notice that in any tiling of~$\G$, there is no $2\times2\times1$ subregion of $\G$
that is tiled by two dominoes, which proves the result.
\end{proof}


\bigskip

\section{Final remarks and open problems}\label{s:fin}

\subsection{}
Historically, the tiling problems played a crucial role in the
developments of modern theoretical computer science.  The tileability
of the whole plane with a set of tiles was shown to be
undecidable~\cite{Ber,Rob}.  A version of the finite tileability problem
was stated to be \NP-complete in Levin's original paper~\cite{Lev},
where he (independently) defined \NP-completeness.  Finally,
Theorem~\ref{thm:domino} is one of the first few applications
of \SP-completeness, developed by Valiant~\cite{Val}.

\subsection{} \label{s:physics}
In the context of statistical physics, the domino tiling problem is
called the \emph{dimer problem}, and has a long history.  The classical
results of Fisher~\cite{Fis} and Kasteleyn~\cite{Kas} express the number
of domino tilings of finite regions in the plane as a certain Pfaffian,
equal to a square root of a determinant.  A closely related
\emph{monomer-dimer model} is \SP-complete already in the
plane~\cite{Jer} (see also~\cite{Vad}).  Let us mention that both
models can be polynomially approximated (see~\cite{JSV,KRS}).
We refer to~\cite{Tho} for graph-theoretic reasons precluding the
Pfaffian method in higher dimensions, and to~\cite{DG,HK} (see also~\cite{HN})
for the general hardness results on \emph{graph homomorphisms},
a concept generalizing perfect matchings.

\subsection{}
In the \emph{general tileability problem}, a set $\T$ of tiles is fixed, and
one considers tilings with parallel translations of copies of tiles $T \in \T$.
For a single tile~$T$, the tileability problem is a special case, where $\T$
consists of all reflections and rotations of~$T$.  Let us briefly elaborate
on the state of art of these tiling problems as they pertain to our results.

In the plane, when $|\T|=1$, \latin{i.e.}, when translates of a single tile are used,
the tileability problem is linear in the area.
However, already for $\T=\{2\times 1, 1\times 3\}$, the tileability is \NP-complete~\cite{BNRR}.
It is not known whether the corresponding counting problem is \SP-complete.
The smallest set~$\T$ for which \SP-completeness is known
is the set of four rotations of the $L$-tromino and a $[2\times 2]$ square~\cite{MR}.

For simply connected regions in the plane, the authors recently found
a set of $23$ Wang tiles, which they showed to be \NP-complete and
\SP-complete~\cite{PY}.  This construction can be further decreased to
$15$ tiles~\cite{Yang}, and the authors conjecture that only three rectangles
suffices (see~\cite{PY}).
In contrast to the decision problem, it is still unclear how much
the simple connectivity of regions helps counting the number of tilings,
other than making the reductions harder and more technical.

\subsection{} \label{ss:fin-val}
In~\cite{Val-alg},
Valiant's goal is to show the hardness of the number of perfect matchings of grid graphs,
defined as subgraphs (of the unit distance graph) of~$\zz^3$.
This is slightly weaker than the hardness of domino tilings of $3$-dim regions,
since the latter problem corresponds to \emph{induced} subgraphs of~$\zz^3$.
However, Valiant's proof can be slightly modified,
from subgraphs of $[2\times n \times n]$ to induced subgraphs of $[3\times n \times n]$,
to achieve the same goal.
The proof we present in Section~\ref{sec-domino} is in fact a variation of the argument in~\cite{Val-alg}.

\subsection{} \label{ss:augment}
For tiling in three dimensions, we consider the domino and the slab.
Here each tile has three distinct orientations.
If we only allow two out of the three orientations in either case,
each problem becomes equivalent to the ordinary domino tiling problem in two dimensions,
and hence is in~\P.
Of course, two orientations suffices for the tromino tile $I_3=[3\times1\times1]$
(or the $[3\times3\times1]$ tile) to guarantee \NP-completeness,
since we have \NP-completeness for the $[3\times1]$ tile.

Note that the \emph{augmentation approach}, mentioned in the introduction,
allows us to show the \NP-completeness of the tileability problem with~$I_3$
of contractible regions in~$\rr^3$.
To see this, start with a flat region and color the holes in a checkerboard fashion.
Now fill each unit cube with an up or down vertical tromino, depending on the color (see \fig{tro}).
The contractible region thus obtained is tileable with~$I_3$
if and only if the starting region is tileable by $[3\times1]$.

\begin{figure}[hbtp]
   \includegraphics[scale=0.47]{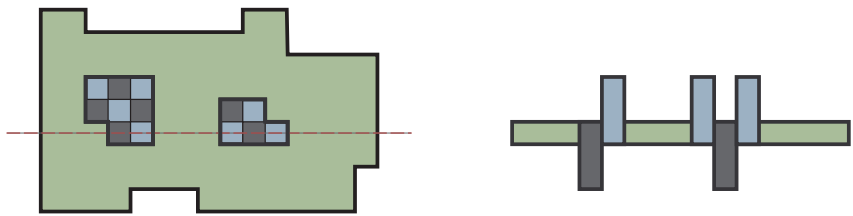}
   \caption{A construction of contractible $3$-dim regions for tromino tilings.}
   \label{fig:tro}
\end{figure}

\subsection{}
The conditions in Theorem~\ref{t:gen} are best possible.
Indeed, if $r=0$ or $r=d$, we are tiling with cubes of side length $1$ or~$2$, respectively,
both of which are linear in the volume.
The only remaining cases are the decision problem with $\cD_1^d$,
and the counting problem of ordinary dominoes in dimension $d=2$,
both of which are in~\P{} (see \latin{e.g.}~\cite{LP}).

\subsection{}
The idea of graph embedding into a grid is quite standard.
In this context of domino tilings it was used in~\cite{DKLM,Val-alg},
and for other small sets of tiles in~\cite{BNRR,MR,PY}.
Of course, the technical details are somewhat different in each case.

\subsection{}
It would be interesting to see if Theorem~\ref{thm:domino-c} holds for
contractible regions inside a $[2\times n \times n]$ brick,
as they have a rather simple geometric structure.
Our proof gives only a width~$4$ bound,
but we believe that a width~$3$ modification can be made without difficulty.
We should mention that for every fixed~$c$,
the counting problem is polynomial for regions inside a $[c\times c \times n]$ brick.

\subsection{}
In the plane, there are several other generalizations of domino tilings.
Notably, the first author introduced and studied \emph{ribbon tilings}
(see~\cite{pak-horizons}).  Another interesting set of generalized dominoes
$\T_n = \{2^k \times 2^{n-k}, 0 \le k \le n\}$ was studied in~\cite{Korn}.
Both sets satisfy the \emph{local move} property: every
two tilings of a s.c.~region~$\G$ can be obtained by a sequence
of \emph{flips}, each involving a fixed number of tiles
($2$~in both cases, see~\cite{pak-horizons}).
For the (usual) domino tilings this is a classical property going back to
Kasteleyn and famously proved by Thurston via the \emph{height functions}~\cite{Thu}.

Of course, for the domino tilings of general regions there is no local
move property, as large cycles can involve an unboundedly many tiles.
However, for contractible regions, one would naturally assume that domino
tilings and slab tilings are connected by flips corresponding to different
tilings of $[2\times 2 \times 2]$.  This is, in fact, false.   As an early
prequel to the constructions in the proofs of theorems~\ref{thm:slab-c}
and~\ref{thm:domino-c}, we state the following easy result.

\begin{prop}\label{c:move}
Domino tilings of contractible regions in~$\rr^3$
does not have the local move property.
The same result holds for slab tilings.
\end{prop}

The proof is apparent from \fig{move} (cf.~\fig{tro}).
In the first case, the middle dominoes alternate between up and down.  In the
second case, slabs are all vertical, with the middle slabs all one layer below.

\begin{figure}[hbtp]
   \includegraphics[scale=0.47]{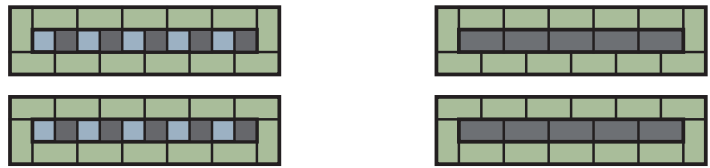}
   \caption[Cross sections of large regions in~$\rr^3$ with exactly two tilings.]
           {Cross sections of large regions in~$\rr^3$ with exactly two tilings, by dominoes and by slabs, respectively.}
   \label{fig:move}
\end{figure}

\vskip.7cm


\noindent
\textbf{Acknowledgements.} \,
We are very grateful to Cris Moore for helpful conversations; in particular,
Theorem~\ref{thm:slab} arose as a question in our discussions.
The first author is partially supported by the NSF and BSF grants.
The second author is supported by the NSF under Grant No.~DGE-0707424.


\newpage

\vskip1.1cm

\begin{thebibliography}{BNRR12}

\bibitem[BNRR]{BNRR}
D.~Beauquier, M.~Nivat, \'{E}.~R\'{e}mila and M.~Robson,
Tiling figures of the plane with two bars,
\emph{Comput. Geom.}~\textbf{5} (1995), 1--25.

\bibitem[Ber]{Ber}
R.~Berger,
The undecidability of the domino problem,
\emph{Mem. AMS}~\textbf{66} (1966), 72~pp.

\bibitem[Cha]{Cha}
T.~Chaboud,
Domino tiling in planar graphs with regular and bipartite dual,
\emph{Theor. Comp. Sci.}~\textbf{159} (1996), 137--142.

\bibitem[CH]{CH}
N.~Creignou and M.~Hermann,
On \SP-completeness of some counting problems,
Research Report~2144, INRIA, 1993.

\bibitem[DL]{DL}
P.~Dagum and M.~Luby,
Approximating the permanent of graphs with large factors,
\emph{Theoret. Comput. Sci.}~\textbf{102} (1992), 283--305.

\bibitem[DKLM]{DKLM}
S.~Datta, R.~Kulkarni, N.~Limaye and M.~Mahajan,
Planarity, determinants, permanents, and (unique) matchings,
\emph{J.~ACM Trans. Comput. Theory}~\textbf{1} (2010), No.~3.

\bibitem[DG]{DG}
M.~Dyer and C.~Greenhill, 
The complexity of counting graph homomorphisms, 
\emph{Random Structures Algorithms}~\textbf{17} (2000), 260--289. 

\bibitem[Fis]{Fis}
M.~E.~Fisher,
Statistical mechanics of dimers on a plane lattice,
\emph{Phys. Rev. (2)}~\textbf{124} (1961), 1664--1672.

\bibitem[HK]{HK}
P.~Hell and D.~G.~Kirkpatrick,
On generalized matching problems,
\emph{Inform. Process. Lett.}~\textbf{12} (1981), 33--35.

\bibitem[HN]{HN}
P.~Hell and J.~Ne\v{s}et\v{r}il,
Counting list homomorphisms for graphs with bounded degrees,
in \emph{DIMACS Ser. Discrete Math. Theoret. Comput. Sci.}~\textbf{63},
AMS, Providence, RI, 2004, 105--112.

\bibitem[Jer]{Jer}
M.~Jerrum,
Two-dimensional monomer-dimer systems are computationally intractable,
\emph{J.~Statist. Phys.}~\textbf{48} (1987), 121--134;
Erratum in~\textbf{59} (1990), 1087--1088.

\bibitem[JSV]{JSV}
M.~Jerrum, A.~Sinclair and E.~Vigoda,
A polynomial-time approximation algorithm for the permanent of a matrix with nonnegative entries,
\emph{J.~ACM}~\textbf{51} (2004), 671--697.

\bibitem[Kas]{Kas}
P.~W.~Kasteleyn,
The statistics of dimers on a lattice~I,
\emph{Physica}~\textbf{27} (1961), 1209--1225

\bibitem[KK]{KK}
C.~Kenyon and R.~Kenyon,
Tiling a polygon with rectangles,
in \emph{Proc. 33rd FOCS} (1992), 610--619.

\bibitem[KRS]{KRS}
C.~Kenyon, D.~Randall and A.~Sinclair,
Approximating the number of monomer-dimer coverings of a lattice,
\emph{J.~Statist. Phys.}~\textbf{83} (1996), 637--659.

\bibitem[Ken]{Ken-dimer}
R.~Kenyon,
An introduction to the dimer model,
in \emph{ICTP Lect. Notes~XVII}, Trieste, 2004.

\bibitem[Korn]{Korn}
M.~Korn,
\emph{Geometric and algebraic properties of polyomino tilings},
MIT Ph.D. thesis, 2004;
available at \url{http://dspace.mit.edu/handle/1721.1/16628}

\bibitem[Lev]{Lev}
L.~Levin,
Universal sorting problems,
\emph{Problems Inf. Transm.}~\textbf{9} (1973), 265--266.

\bibitem[LP]{LP}
L.~Lov\'{a}sz and M.~D.~Plummer,
\emph{Matching theory},
AMS, Providence, RI, 2009.

\bibitem[MR]{MR}
C.~Moore and J.~M.~Robson,
Hard tiling problems with simple tiles,
\emph{Discrete Comput. Geom.}~\textbf{26} (2001), 573--590.

\bibitem[Pak]{pak-horizons}
I.~Pak,
Tile invariants: New horizons,
\emph{Theor. Comp. Sci.}~\textbf{303} (2003), 303--331.

\bibitem[PY]{PY}
I.~Pak and J.~Yang,
Tiling simply connected regions with rectangles,
\emph{J.~Combin. Theory, Ser.~A}, to appear.

\bibitem[R\'{e}m]{Rem2}
\'{E}.~R\'{e}mila, Tiling a polygon with two kinds of rectangles,
\emph{Discrete Comput. Geom.}~\textbf{34} (2005), 313--330.

\bibitem[Rob]{Rob}
R.~M.~Robinson,
Undecidability and nonperiodicity for tilings of the plane,
\emph{Invent. Math.}~\textbf{12} (1971), 177--209.

\bibitem[Sch]{Sch}
T.~Schaefer,
The complexity of satisfiability problems,
in \emph{Proc.~10th STOC} (1978), 216--226.

\bibitem[Tho]{Tho}
R.~Thomas,
A survey of Pfaffian orientations of graphs,
in \emph{Proc. ICM Madrid},
Vol.~III, 963--984, Z\"urich, 2006.

\bibitem[Thu]{Thu}
W.~P.~Thurston,
Conway's tiling groups,
\emph{Amer. Math. Monthly}~\textbf{97} (1990), 757--773.

\bibitem[Til]{Til}
B. de Tili\`{e}re,
\emph{The dimer model in statistical mechanics},
Lecture notes at Swiss Doctoral Program in Mathematics and the EPFL doctoral school,
University of Neuch\^atel, September 2008;
available at \url{http://proba.jussieu.fr/~detiliere/Cours/polycop_Dimeres.pdf}

\bibitem[Vad]{Vad}
S.~P.~Vadhan,
The complexity of counting in sparse, regular, and planar graphs,
\emph{SIAM J. Comput.}~\textbf{31} (2001), 398--427.

\bibitem[Val1]{Val}
L.~G.~Valiant,
The complexity of enumeration and reliability problems,
\emph{SIAM J. Comput.}~\textbf{8} (1979), 410--421.

\bibitem[Val2]{Val-alg}
L.~G.~Valiant,
Completeness classes in algebra,
in \emph{Proc. 11th STOC} (1979), 249--261.

\bibitem[Yang]{Yang}
J.~Yang,
Ph.D.~thesis,
in preparation.

\end{thebibliography}
\end{document}